\documentclass{article}
\usepackage{amsmath}
\usepackage{amssymb}
\usepackage[english] {babel}

\def\a{\alpha} \def\b{\beta} \def\d{\delta} \def\e{\epsilon} \def\f{\varphi}
\def\l{\lambda} \def\s{\sigma} \def\R{\mathbb{R}} \def\C{\mathbb{C}}
 \def\I{\mathbb{I}}
 \def\minus{\smallsetminus}
\def\({\left(} \def\){\right)}
\def\<{\langle} \def\>{\rangle}
 \def\bx{\hspace*{\fill}$\Box$\vspace{1ex}}

\def\N{\mathbb{N}}

\def\vc{{V^{\C}}}
\def\dc{\d_\C} \def\ec{\e_\C}

\newenvironment{proof}{\textit{Proof:}}{\bx}

\newcommand\ts{\otimes}

\newcommand{\ie}{i.e.}

\newcommand{\cf}{c.f.}
\renewcommand{\phi}{\varphi}

\renewcommand{\le}{\leqslant}

\DeclareMathOperator{\spann}{span}

\DeclareMathOperator{\End}{End}
\DeclareMathOperator{\og}{O}
\DeclareMathOperator{\rep}{rep}

\newtheorem{thm}{Theorem}[section]
\newtheorem{lma}[thm]{Lemma}
\newtheorem{prop}[thm]{Proposition}

\title{Classification of pairs of rotations in finite-dimensional Euclidean space}

\author{Erik Darp\"{o}\\ {\small \textit{Matematiska Institutionen, Uppsala
  Universitet,}}\\{\small\textit{Box 480, S-75106 Uppsala,
  Sweden.}}\\{\small\texttt{erik.darpo@math.uu.se}}}  

\begin{document}
\selectlanguage{english}
\date{} 
\maketitle 

\textit{Dedicated to Fred Van Oystaeyen, on the occasion of his sixtieth birthday}

\begin{abstract}
A rotation in a Euclidean space $V$ is an orthogonal map $\d\in\og(V)$ which acts locally
as a plane rotation with some fixed angle $a(\d)\in[0,\pi]$. 
We give a classification of all finite-dimensional representations of the real algebra
$\R\<X,Y\>$ which are given by rotations, up to orthogonal isomorphism.
\end{abstract}

\noindent
Mathematics Subject Classification 2000: 15A21 (16G20, 15A57, 17A80) \\
Keywords: Rotation, rotational representation, irreducible representation, invariant
subspace, classification.

\section{Definitions and results} \label{intro}

Let $V=(V,\<\,\>)$ be a Euclidean space. A \emph{rotation} in $V$ is an orthogonal
endomorphism $\d\in\og(V)$ such that
\begin{enumerate}
\item $\d^2(v)\in\spann\{v,\d(v)\}$ and
\item $\<v,\d(v)\>=\<w,\d(w)\>$
\end{enumerate}
for all unit vectors $v,w\in V$.\footnote{The first axiom is redundant if $V$ is
  finite-dimensional. An orthogonal map satisfying only 1 is either a rotation
  or a reflection in a proper, non-trivial subspace of $V$.}
By the \emph{angle} of $\d$ then is meant the number
$a(\d)=\arccos\<v,\d(v)\>\in[0,\pi]$. If $a(\d)\in\{0,\pi\}$, then $\d=\pm\I_V$, where
$\I_V$ denotes the identity map on $V$. If $a(\d)\in]0,\pi[$, we call $\d$ a \emph{proper}
rotation. 

Given any orthogonal operator $\s$ on a finite-dimensional Euclidean space $V$, there
exists an orthonormal basis $\underline{e}$ of $V$ such that the matrix of $\s$ with
respect to $\underline{e}$ has the form
\begin{equation} \label{ortnf}
  [\s]_{\underline{e}}=R_{\a_1}\oplus\cdots\oplus R_{\a_k}\oplus\I_l\oplus-\I_m
  \mbox{ with }
  \a_i\in]0,\pi[ \mbox{ and }k,l,m\in\N
\end{equation}
where
\begin{equation*}
  R_\a=\begin{pmatrix} \cos\a & -\sin\a \\ \sin\a & \cos\a \end{pmatrix}\in\og_2(\R),
  \quad
  A\oplus B=
  \begin{pmatrix}  A\\&B  \end{pmatrix}
\end{equation*}
and $\I_n$ is the identity matrix of size $n\times n$.
The presentation in (\ref{ortnf}) is unique up to permutations of the summands
$R_{\a_i}$.
This \emph{structure theorem for orthogonal operators} \cite{gantmacher} plays an
important role in the present article.

In view of the above description, the map $\s$ is a proper rotation if and only if
there exists an orthonormal basis $\underline{e}$ of $V$ such that
$[\s]_{\underline{e}}=R_\a\oplus\cdots\oplus R_\a$ for some $\a\in]0,\pi[$.

The structure theorem provides complete information about the behaviour of a single
orthogonal operator on $V$. Our main result, formulated in the three propositions below,
gives a corresponding picture for pairs $(\d,\e)$ of rotations in a finite-dimensional
Euclidean space.

\begin{prop} \label{decomp}
  Let $\d$ and $\e$ be rotations in a finite-dimensional Euclidean space $V$.
  The space $V$ decomposes into an orthogonal direct sum of subspaces, each of which is
  invariant under $\d$ and $\e$, and has dimension either 1, 2 or 4.
\end{prop}

Every pair $(\s,\tau)$ of linear endomorphisms of a real vector space $V$ gives rise to a
representation in $V$ of the free associative algebra $\R\<X,Y\>$ with generators $X,Y$,
via the algebra morphism $\R\<X,Y\>\to \End(V)$ determined by 
$X\mapsto\s,\;Y\mapsto\tau$. We denote this representation again by $(\s,\tau)$.
If $V$ is Euclidean and $\d,\e$ are rotations in $V$, we call $(\d,\e)$ a
\emph{rotational} representation. We denote by $\mathfrak{R}$ the category of all
finite-dimensional rotational representations of $\R\<X,Y\>$.
Given rotations $\d,\e\in\og(V)$ and $\s,\tau\in\og(W)$, a morphism
$(\d,\e)\to(\s,\tau)$ in $\mathfrak{R}$ is a linear map $\f:V\to W$ which decomposes as
$\f=\psi\oplus0:U\oplus\ker\f\to W$ where $U=(\ker\f)^\perp\subset V$ is the orthogonal
complement of $\ker\f$, and the map $\psi:U\to W$ is orthogonal.
Thus $\mathfrak{R}$ is a non-full subcategory of $\rep_f(\R\<X,Y\>)$, the
category of finite-dimensional representations of $\R\<X,Y\>$.

We remark that $\mathfrak{R}$ is closed under direct summands, but not under direct sums:
Any subrepresentation of a rotational representation is again rotational. However, the
direct sum of two rotations with different angles is not a rotation, and therefore, the
sum of two objects in $\mathfrak{R}$ is in general not a rotational representation.

If $\s\in\og(V)$ and $U\subset V$ is an invariant subspace for $\s$, then the orthogonal
complement $U^\perp\subset V$ of $U$ is also invariant under $\s$.
This means that if $\d,\e$ are rotations in $V$ and $U\subset V$ is invariant under $\d$
and $\e$, then the representation $(\d,\e)$ decomposes as
$(\d,\e)=(\d|_U,\e|_U)\oplus(\d|_{U^\perp},\e|_{U^\perp})$.
Thus a rotational representation of $\R\<X,Y\>$ is indecomposable if and only if it is
irreducible, and every finite-dimensional rotational representation can be decomposed into
a direct sum of irreducible representations. 
Moreover, the next proposition asserts that this decomposition is essentially unique.

\begin{prop} \label{ks}
  The category $\mathfrak{R}$ has the Krull-Schmidt property:
  If $(\d,\e)\in\mathfrak{R}$, and $(\d,\e)=(\s_1,\tau_1)\oplus\cdots\oplus(\s_k,\tau_k)$
  and $(\d,\e)=(\s'_1,\tau'_1)\oplus\cdots\oplus(\s'_l,\tau'_l)$ are decompositions of
  $(\d,\e)$ into irreducible subrepresentations, then $k=l$ and there exists a permutation
  $f\in S_k$ such that $(\s_{f(i)},\tau_{f(i)})$ is isomorphic to $(\s'_i,\tau'_i)$ for
  all $i\le k$.
\end{prop}

The above statement is not trivial. Since $\mathfrak{R}$ is not a
\emph{full} subcategory of $\rep_f(\R\<X,Y\>)$, isomorphism classes in the former are
\`a priori smaller than in the latter. Therefore the Krull-Schmidt theorem,
although certainly true in $\rep_f(\R\<X,Y\>)$, does not automatically carry over to
$\mathfrak{R}$.

A consequence of Propositions~\ref{decomp} and~\ref{ks} is the following:
If $V$ is a finite-dimensional Euclidean space, then every pair $(\d,\e)$ of rotations in
$V$ has a has a 2-dimensional invariant subspace $U\subset V$ if and only if $\dim V\equiv
2\mod4$. This result has been proven independently by Dieterich~\cite{rotations}, using
determinant calculus.

We complete the investigation of the category $\mathfrak{R}$ by giving a classification of
its irreducible objects. In view of Proposition~\ref{ks}, this amounts to classifying
$\mathfrak{R}$ itself.
By a classification is meant a list of pairwise non-isomorphic objects, exhausting all
isomorphism classes.

Let $\mathfrak{IR}$ be the full subcategory of $\mathfrak{R}$ consisting of all
irreducible finite-dimensional rotational representations.
We consider $\R^n$ as a Euclidean space, equipped with the standard scalar product. Linear
maps $\R^n\to\R^m$ are identified with $m\!\times\!n$\,-\,matrices in the natural way.
Given $\theta\in\R$, we write
\begin{equation} \label{ttheta}
  T_\theta= \begin{pmatrix} 1\\&R_\theta\\&&1 \end{pmatrix}\in\og_4(\R).
\end{equation}
\begin{prop} \label{class}
  The category $\mathfrak{IR}$ is classified by the following list of objects:
  \begin{align*}
    &(r,s)\quad &\mbox{where }\;&r,s\in\{-1,1\}, \\
    &(r\I_2,R_\b), (R_\a,s\I_2), (R_\a,R_{r\b})  &\mbox{where }
    \;&r,s\in\{-1,1\},\;\a,\b\in]0,\pi[, \\ 
    &\left(R_\a\oplus R_\a,
	T_\theta(R_\b\oplus R_\b)T_{-\theta} \right)
    &\mbox{where }\;& \a,\b,\theta\in]0,\pi[.
  \end{align*}
\end{prop}

Our results indicate the contrast between the categories $\rep_{f}(\R\<X,Y\>)$ and
$\mathfrak{R}$.
The algebra $\R\<X,Y\>$ is wild, and the number of parameters of non-iso\-mor\-phic
indecomposable representations increases heavily with the dimension. By
Proposition~\ref{decomp}, $\mathfrak{R}$ has no indecomposable objects of dimension
greater than 4. 
The category $\rep_{f}(\R\<X,Y\>)$ is immense, the objects in $\mathfrak{R}$ are easily
visualised with geometric intuition.

The author's interest in rotational representations originates in the theory of real
(not necessarily associative) division algebras.
Proposition~\ref{class} plays an important role in the classification of the 8-dimensional
absolute valued algebras which have either a non-zero central idempotent or a one-sided
identity element \cite{ava}.
This is an instance of what appears to be a general pattern: The connection between
various classes of real division algebras and some ``geometric'' subcategory of the
representation category of a real associative algebra determined by the class of division
algebras.
As examples can be mentioned division algebras of dimension 2, which are related to
certain representations of $\R\<X,Y\>$, and 8-dimensional flexible quadratic division
algebras, which are parametrised by a non-full subcategory of $\rep_{f}(\R[X])$ (\cf\
\cite{2dim} and \cite{nform} respectively).

In Section~\ref{sdecomp} below, we prove Proposition~\ref{decomp}.
In addition a criterion is obtained, formulated in Proposition~\ref{clmt}, for when a pair
of rotations has a 2-dimensional invariant subspace.
In Section~\ref{sclass}, we show that Proposition~\ref{ks} can be deduced from the
Krull-Schmidt theorem for $\rep_f(\R\<X,Y\>)$. Finally, by classifying pairs of rotations
in Euclidean spaces of dimension at most 4, we attain the proof of Proposition~\ref{class}.

Henceforth, $V=(V,\<\,\>)$ will always denote a finite-dimensional Euclidean space.

\section{Decomposition of $V$ into invariant subspaces} \label{sdecomp}

Let $\d$ and $\e$ be rotations in $V$, with respective angles $\a$ and $\b$. 
We may assume that both $\d$ and $\e$ are proper rotations, otherwise the decomposition of
$V$ into 1- or 2-dimensional invariant subspaces follows directly from the structure
theorem for orthogonal operators. We say that a subspace $U\subset V$ is
$(\d,\e)$-invariant if it is invariant under both $\d$ and $\e$.

Consider the complexification $\vc=\C\ts_\R V$ of $V$ equipped with the induced inner
product: $\<\l\ts u,\mu\ts v\>=\l\bar{\mu}\<u,v\>$. We identify $V$ with the real subspace
$\R\ts V$ of $\vc$ in the canonical way. Elements in this subspace are called \emph{real}
vectors, whereas elements in $iV=\R i\ts V$ are said to be \emph{purely imaginary}.
As a real vector space, $\vc=V\oplus iV$. The complex conjugate of a vector 
$v=v_1+v_2\in V\oplus iV$ is $\bar{v}=v_1-v_2$, and the conjugation map
$\kappa:\vc\to\vc,\;v\mapsto \bar{v}$ is an antilinear\footnote{A map $T:V\to W$ between
  complex vector spaces is called antilinear if it is additive and $T(\l v)=\bar{\l}Tv$
  for all $\l\in\C,\;v\in V$. A theory for antilinear maps is developed in
  \cite{nakayama38,haantjes36,hong88,jac37,jac39}.}
operator on $\vc$. 
By $T_\C$ we denote the $\C$-linear endomorphism of $\vc$ induced by a real linear
endomorphism $T$ of $V$. 
We will make use of the fact that the conjugation map $\kappa$ induces bijections
$\ker(T_\C-\l\I_\vc)\to\ker(T_\C-\bar{\l}\I_\vc)$ for all eigenvalues $\l\in\C$ of $T_\C$.

The maps $\d_\C$ and $\e_\C$ give rise to decompositions $\vc=A\oplus B$ and 
$\vc=C\oplus D$, where
\begin{equation} \label{abcd}
\begin{aligned}
A&=\ker(\dc-e^{i\a}\I_\vc) & C&=\ker(\ec-e^{i\b}\I_\vc) \\
B&=\ker(\dc-e^{-i\a}\I_\vc) & D&=\ker(\ec-e^{-i\b}\I_\vc)
\end{aligned}
\end{equation}
and the summands in each decomposition are mutually orthogonal. As noted above,
$\kappa(A)=B$ and $\kappa(C)=D$.

Suppose $A\cap C$ is non-trivial, and $v\in(A\cap C)\minus\{0\}$. 
Then $\bar{v}\in B\cap D$, and $B\cap D\ne0$.
Hence both $v$ and $\bar{v}$ are common eigenvectors of $\dc$ and $\ec$, which in
particular means that $\spann_\C\{v,\bar{v}\}$ is invariant under $\dc$ and $\ec$.
On the other hand, $v+\bar{v},i(v-\bar{v})$ are real vectors, and 
$\spann_\C\{v+\bar{v},i(v-\bar{v})\}=\spann_\C\{v,\bar{v}\}$.
Hence $\spann_\R\{v+\bar{v},i(v-\bar{v})\}=V\cap\spann_\C\{v,\bar{v}\}$. This is a
2-dimensional (real) subspace of $V$, invariant under $\dc|_V=\d$ and $\ec|_V=\e$.

Certainly, the above argument goes through also when the roles of $C$ and $D$ are
interchanged. Thus we have shown, that whenever $A\cap(C\cup D)\ne0$, $V$ has a
2-dimensional subspace which is invariant under $\d$ and $\e$. 

The crucial property of the 2-dimensional subspace $\spann_\C\{v,\bar{v}\}\subset \vc$
above is, that its intersection with each one of the sets $A,B,C,D,V\subset\vc$ is
non-zero. This is precisely what is needed for $V$ to have a 2-dimensional
$(\d,\e)$-invariant subspace. More generally, the following holds.

\begin{lma} \label{U}
  If $U\subset \vc$ is a $\kappa$-invariant subspace such that
  \begin{equation} \label{Uekv}
    U=(A\cap U)\oplus(B\cap U)=(C\cap U)\oplus(D\cap U)
  \end{equation}
  then $U\cap V\subset V$ is invariant under $\d$ and $\e$, and 
  $\dim_\R(U\cap V)=\dim_\C U$.
  Conversely, if $W\subset V$ is an invariant subspace for $\d$ and $\e$, then its
  complexification $U=W^\C$ is a $\kappa$-invariant subspace satisfying (\ref{Uekv}).
\end{lma}

\begin{proof}
Suppose $U\subset \vc$ is a $\kappa$-invariant subspace satisfying (\ref{Uekv}), and let
$\underline{b}$ be a basis of $A\cap U$. Then 
$\underline{b}'=\{\bar{u}\mid u\in\underline{b}\}$ is a basis of $B\cap U$ and
$\underline{b}\cup\underline{b}'$ a basis of $U$. 
Taking $\underline{e}=\{u+\bar{u}, i(u-\bar{u})\mid u\in\underline{b}\}$, we get
$\spann_\C\underline{e}=\spann_\C(\underline{b}\cup\underline{b}')=U$. Hence
$\underline{e}$ is a basis of $U$.
Moreover, $\underline{e}\subset U\cap V$. A vector $v\in\spann_\C\underline{e}=U$
is real if and only if its coefficients in $\underline{e}$ are real numbers. Thus
$\spann_\R\underline{e}=U\cap V$ and consequently $\dim_\R U\cap V=\dim_\C U$.

The real space $V$ is invariant under $\dc,\ec$, and from (\ref{Uekv}) follows that so is
$U$. Therefore, the intersection $U\cap V$ is invariant under $\d=\dc|_V$ and
$\e=\ec|_V$.

The converse is immediate.
\end{proof}

In the case when $A\cap(C\cup D)=0$, $\vc$ may or may not have a 2-dimensional
$\kappa$-invariant subspace $U$ satisfying (\ref{Uekv}).
Below we construct a map $T$, which will be our tool to find subspaces of this type, of
dimension either 2 or 4.

Assume $A\cap C=A\cap D=0$. Let $P_A:\vc\to A$ and $P_B:\vc\to B$ be the orthogonal
projection maps onto $A$ and $B$ respectively. Since the subspaces $\ker P_A=B$ and 
$\ker P_B=A$ intersect $C$ trivially, the restricted maps $P_A|_C:C\to A$ and
$P_B|_C:C\to B$ are bijective. 
We define $T:A\to A$ as the composition of the following chain of maps:
\begin{equation} \label{T}
  A\stackrel{P_A|_C^{-1}}{\to}C\stackrel{P_B|_C}{\to}B\stackrel{\kappa}{\to}A
\end{equation}
\ie, $T=\kappa\circ P_B|_C\circ P_A|_C^{-1}$.
Clearly, $T$ being the composition of two linear and one antilinear map, itself is
antilinear. The factors of $T$ are bijections, hence $T$ is a bijection.

\begin{prop} \label{clmt}
  Let $V$ be a non-trivial Euclidean space of finite dimension, and $\d,\e$ rotations in
  $V$ with angles $\a,\b\in]0,\pi[$ respectively. Let $A,B,C,D\subset\vc$ be the subspaces
  defined by (\ref{abcd}). Now there exists a 2-dimensional subspace of $V$ which is
  invariant under $\d$ and $\e$ if and only if one of the following conditions is
  satisfied: 
  \begin{enumerate}
  \item Either of the intersections $A\cap C$ and $A\cap D$ is non-trivial.
  \item Both $A\cap C$ and $A\cap D$ are zero, and the antilinear operator 
    $T:A\to A$ defined by (\ref{T}) has a 1-dimensional invariant subspace.
  \end{enumerate}
\end{prop}

We remark that an antilinear operator $S$ has a 1-dimensional invariant subspace if and
only if $S^2$ has a non-negative real eigenvalue. Moreover, this is always the case if the
domain of $S$ is odd-dimensional. This is a consequence of the normal form
for antilinear operators given in \cite{haantjes36}.

\begin{proof}
We have already shown, that the first alternative in the proposition implies existence of
a 2-dimensional $(\d,\e)$-invariant subspace. 
Consider instead the case $A\cap(C\cup D)=0$. 

Assume $T$ has a 1-dimensional invariant subspace. This means that there exists a non-zero
vector $v\in A$ such that $Tv=\mu v$ for some $\mu\in\C$.
Setting $w=P_A|_C^{-1}v$, we have $P_Aw=v$ and $P_Bw=P_B|_CP_A|_C^{-1}v=\kappa
Tv=\bar{\mu}\bar{v}$. Hence $w=P_Aw+P_Bw=v+\bar{\mu}\bar{v}$. Since $w\in C$, this means
that $C\cap\spann_\C\{v,\bar{v}\}\neq0$. Now 
$\bar{w}=\kappa(w)=\bar{v}+\mu v\in D\cap\spann_\C\{v,\bar{v}\}$, whence
$D\cap\spann_\C\{v,\bar{v}\}\neq0$.

So $U=\spann_\C\{v,\bar{v}\}$ intersects each of $A,B,C,D$ non-trivially. 
Hence $U=(A\cap U)\oplus(B\cap U)=(C\cap U)\oplus(D\cap U)$. Obviously $\kappa(U)=U$, and
Lemma~\ref{U} now asserts that $U\cap V$ is a 2-dimensional subspace of $V$, invariant
under $\d$ and $\e$.

Assume instead there exists a real subspace $W\subset V$ of dimension 2, which is
invariant under $\d$ and $\e$. By Lemma~\ref{U}, $\kappa(W^\C)=W^\C$ and 
$W^\C=(A\cap W^\C)\oplus(B\cap W^\C)=(C\cap W^\C)\oplus(D\cap W^\C)$. This implies that
$W^\C$ is invariant under $P_A$, $P_B$ and $P_A|_C^{-1}$ (the last of which is defined
since we have assumed $A\cap(C\cup D)=0$). Thus $A\cap W^\C$ is invariant under $T$. 
As $\dim(A\cap W^\C)=1$, we are done.
\end{proof}

Along the same lines, we can now prove Proposition~\ref{decomp}. Instead of the
antilinear map $T$, which not necessarily has a 1-dimensional invariant subspace, we
consider the linear map $T^2:A\to A$.

Since every object in $\mathfrak{R}$ can be written as a sum of irreducibles, it suffices
to show that for any two rotations $\d,\e$ in $V$, there exists a $(\d,\e)$-invariant
subspace $W\subset V$ of dimension at most 4.
We may assume that $A\cap(C\cup D)=0$, since otherwise $V$ has a $(\d,\e)$-invariant
subspace of dimension 2. Now $T:A\to A$ is defined, and $T^2$ is a linear endomorphism of
the complex vector space $A$. Thus it has an eigenvalue $\l\in\C$. Let $u\in A$ be a
corresponding, non-zero eigenvector. Set $v=Tu$ and $w=P_A|_C^{-1}u$. 
We get $\kappa P_B w=\kappa(P_B|_C)(P_A|_C^{-1})u=Tu=v$ and $P_Av=u$, so $w=u+\bar{v}$.
Similarly, on setting $z=P_A|_C^{-1}v$ we obtain 
\begin{equation*}
  \kappa P_B z=\kappa(P_B|_C)(P_A|_C^{-1})v=\kappa(P_B|_C)(P_A|_C^{-1})Tu=T^2u=\l u.
\end{equation*}
Thus $P_Bz=\bar{\l}\bar{u}$ and $z=P_Az+P_Bz=v+\bar{\l}\bar{u}$.

If $u$ and $v$ are linearly dependent, then $\spann_\C\{u\}$ is an invariant subspace for $T$
of dimension 1 and hence, by Proposition~\ref{clmt}, $V$ has a 2-dimensional
$(\d,\e)$-invariant subspace.

If $u,v$ are linearly independent, then so are $w,z$. Set
$U=\spann_\C\{u,v,\bar{u},\bar{v}\}$. Since now $w,z\in C\cap U$, we have
$\bar{w},\bar{z}\in D\cap U$. Thus $\dim_\C(C\cap U)=\dim_\C(D\cap U)=2$ and 
$U=(C\cap U)\oplus(D\cap U)$. Moreover $U=(A\cap U)\oplus(B\cap U)$, since 
$u,v\in A\cap U$ and $\bar{u},\bar{v}\in B\cap U$. 
Lemma~\ref{U} now implies that $U\cap V$ is invariant under $\d$ and $\e$, and
$\dim_\R(U\cap V)=\dim_\C U=4$.
This proves Proposition~\ref{decomp}.

\section{Classification of the finite-dimensional rotational representations of $\R\<X,Y\>$}
\label{sclass}
Denote by $\mathfrak{IR}_p$ and $\mathfrak{IR}_o$ the full subcategories of
$\mathfrak{IR}$ consisting of representations given, respectively, by proper rotations and
rotations with angle $\frac{\pi}{2}$.
If $\mathcal{A}$ is a subcategory of $\mathfrak{IR}$ and $n$ a natural number,
$(\mathcal{A})_n$ denotes the full subcategory of $\mathcal{A}$ formed by its
$n$-dimensional objects.
We write $\mathcal{R}$ for the set of
proper rotations in $V$, and $\mathcal{R}_\a$ for the set of rotations in $V$ with angle $\a$.
For $\d\in\mathcal{R}$, we define the map $\rho_\d:V\to V$ by
$\rho_\d(v)=\frac{1}{\sin\a}P_{v^\perp}\d(v)=\frac{1}{\sin \a}(\d(v)-\cos\a\,v)$.

Let $(\d_i,\e_i), (\d'_j,\e'_j) \in \mathfrak{IR}$ for $i\le k$ and $j\le l$.
Suppose 
\begin{equation*}
  \bigoplus_{i=1}^k(\d_i,\e_i) \simeq \bigoplus_{j=1}^l(\d'_j,\e'_j)
\end{equation*}
in $\rep_f(\R(\<X,Y\>)$.
Then, because of the Krull-Schmidt theorem, $k=l$ and there exists
a permutation $f\in S_k$ and isomorphisms $\f_i:(\d_{f(i)},\e_{f(i)})\to(\d'_i,\e'_i)$ in
$\rep_f(\R\<X,Y\>)$ for all $i\le k$.
The following proposition shows that any pair of irreducible rotational representations
which are isomorphic in $\rep_f(\R(\<X,Y\>)$ are also isomorphic in $\mathfrak{R}$. In
view of the above, this establishes the proof of Proposition~\ref{ks}. 

\begin{prop} \label{lambdafi}
  Let $(\d,\e),(\s,\tau)\in\mathfrak{IR}$. If $\f:(\d,\e)\to(\s,\tau)$ is an isomorphism
  in $\rep_f(\R(\<X,Y\>)$, then there exists a number $\l\in\R$ such that $\l\f$ is
  an orthogonal map (and thus an isomorphism in $\mathfrak{R}$).
\end{prop}

\begin{proof}
If $\f:(\d,\e)\to(\s,\tau)$ is such an isomorphism, then $\s=\f\d\f^{-1}$ and $\tau=\f\e\f^{-1}$.
The complexification $\dc$ of $\d$ has eigenvalues $e^{ia(\d)}$ and $e^{-ia(\d)}$. This
implies that $a(\d)\in[0,\pi]$ is determined by the spectrum of $\dc$, which is invariant
under conjugation with $\f$. The same certainly holds for $a(\e)$.
So $a(\s)=a(\d)$ and $a(\tau)=a(\e)$. 

Now, as $\rho_\d=\frac{1}{\sin a(\d)}(\d-\cos a(\d)\,\I)$, we have
$\rho_\s\f=\f\rho_\d$. If $v$ is any vector in the space $V$ carrying $(\d,\e)$, then
$\<\f(v),\f(\rho_\d(v))\>=\<\f(v),\rho_\s(\f(v))\>=0$. This means that $\f$ restricted to
$\spann\{v,\rho_\d(v)\}$ preserves orthogonality, or equivalently, that there exists a
number $\mu\in\R\minus\{0\}$ such that $\|\f(u)\|=\mu\|u\|$ for all
$u\in\spann\{v,\rho_\d(v)\}$.

A similar argument shows that $\|\f(u)\|=\mu\|u\|$ whenever $u\in\spann\{v,\rho_\e(v)\}$.
Indeed, this extends to be true for all $u$ in $\R\<\d,\e\>v$, the subrepresentation of
$(\d,\e)$ generated by $v$. Since $(\d,\e)$ is irreducible, $\R\<\d,\e\>v=V$ if
$v\neq0$. Consequently, $\|\f(u)\|=\mu\|u\|$ for all $u\in V$, and thus the map
$\frac{1}{\mu}\f$ is orthogonal. 
\end{proof}

We proceed to prove Proposition~\ref{class}.

\begin{lma} \label{rho}
\begin{enumerate}
\item For any proper rotation $\d$ in $V$, the map $\rho_\d$ is a rotation with angle
  $\frac{\pi}{2}$. 
\item The map $\mathfrak{d}:\mathcal{R}\to\mathcal{R}_{\pi/2}\times]0,\pi[,\;
\d\mapsto(\rho_\d,a(\d))$ is a bijection.
\item Suppose $\d,\e\in\mathcal{R}$. The representation $(\d,\e)$ is irreducible if and
  only if $(\rho_\d,\rho_\e)$ is irreducible.
\item If $\mathcal{C}\subset\mathcal{R}_{\pi/2}^2$ classifies $\mathfrak{IR}_o$, then 
$\{(\mathfrak{d}^{-1}(\s,\a),\mathfrak{d}^{-1}(\tau,\b))\}_
{(\s,\tau)\in\mathcal{C},\:\a,\b\in]0,\pi[}$ classifies $\mathfrak{IR}_p$.
\end{enumerate}
\end{lma}

\begin{proof}
Let $\d\in\mathcal{R}_\a$.
Clearly, $\rho_\d$ is linear.
Since $\d(v)=\cos\a\,v + P_{v^\perp}\d(v)$, we have
$\|P_{v^\perp}\d(v)\|=\sin\a\|v\|$ and thus
$\|\rho_\d(v)\|=\|v\|$. This means that $\rho_\d$ is orthogonal.
Since $\<v,\rho_\d(v)\>=0$ for all $v\in V$, indeed
$\rho_\d\in\mathcal{R}_{\pi/2}$.

Let $\s\in\mathcal{R}_{\pi/2}$. There exists a basis $\underline{e}$ of $V$ such that
$[\s]_{\underline{e}}=R_{\frac{\pi}{2}}\oplus\cdots\oplus R_{\frac{\pi}{2}}$.
The preimage of $\s$ under the map $\mathcal{R}\to\mathcal{R}_{\pi/2},\;\d\mapsto\rho_\d$
is the set of $\d\in\mathcal{R}$ for which 
$[\d]_{\underline{e}}=R_{a(\d)}\oplus\cdots\oplus R_{a(\d)}$. Hence, for fixed
$\a\in]0,\pi[$ there exists precisely one $\d\in\mathcal{R}_\a$ such that
$\rho_\d=\s$. This proves the second statement in the lemma.

We have seen (Proposition~\ref{decomp}) that every irreducible rotational representation
of $\R\<X,Y\>$ has dimension at most 4. On the other hand, every 2-dimensional
representation given by proper rotations is necessarily irreducible. 
Thus, to prove 3 remains only the 4-dimensional case.
Let $\dim V=4$ and $\d,\e\in\mathcal{R}$.
The representation $(\d,\e)$ is irreducible if
and only if there exists a non-zero vector $v\in V$ such that $\d(v)$ and $\e(v)$ are
linearly independent. This is equivalent to that $\rho_\d(v)$ and $\rho_\e(v)$ are
linearly independent, which happens if and only if the representation
$(\rho_\d(v),\rho_\e(v))$ is irreducible.

Number 4 is an immediate consequence of 2 and 3.
\end{proof}

The virtue of Lemma~\ref{rho} is that it effectively reduces the classification problem
for $\mathfrak{IR}$ to the corresponding problem for $\mathfrak{IR}_o$.

From the structure theorem for orthogonal endomorphisms immediately follows, that the
finite-dimensional irreducible rotational representations which are not in
$\mathfrak{IR}_p$ are classified by 
\begin{align*}
  &\{(r,s)\}_{r,s\in\{-1,1\}}\;\cup \\
  &\{(r,R_\a), (R_\a,r)\}_{r\in\{-1,1\},\:\a\in]0,\pi[}.
\end{align*}

Let $V$ be of dimension 2, and $\s,\tau\in\mathcal{R}_{\pi/2}$. For
any $v\in V\minus\{0\}$ we have $\s(v),\tau(v)\in v^\perp$ and hence $\s(v)=r\tau(v)$,
where $r=\pm1$. Since $\s(v)=\tau(v)$ implies $\s=\tau$, $r$ does not depend on $v$.
Thus, every 2-dimensional object in $\mathfrak{IR}_o$ is isomorphic to either
$(R_{\frac{\pi}{2}},R_{\frac{\pi}{2}})$ or $(R_{\frac{\pi}{2}},R_{-\frac{\pi}{2}})$.
By Lemma~\ref{rho}, this means that
\begin{equation*}
\{(R_\a,R_\b),(R_\a,R_{-\b})\}_{\a,\b\in]0,\pi[}
\end{equation*}
classifies $(\mathfrak{IR}_p)_2$.

It remains to consider the 4-dimensional case. Suppose $\dim V=4$,
$\s,\tau\in\mathcal{R}_{\pi/2}$, and that the representation $(\s,\tau)$ is irreducible.
We shall construct a basis of $V$, with respect to which the matrices of $\s$ and $\tau$
take certain, canonical forms.
Let $e_1\in V$ be any unit vector, and $e_2=\s(e_1)$. Since $(\s,\tau)$ is irreducible,
$\tau(e_1)\not\in\spann\{e_1,e_2\}$. Take
$e_3=\frac{1}{\sqrt{1-\<e_2,\tau(e_1)\>^2}}P_{e_2^\perp}\tau(e_1)$. This is a unit vector
orthogonal to $e_1,e_2$. Finally, by setting $e_4=\s(e_3)$ we get an orthonormal basis
$\underline{e}$ of $V$, such that
\begin{equation} \label{sigma}
[\s]_{\underline{e}}=R_{\frac{\pi}{2}}\!\oplus\!R_{\frac{\pi}{2}}.
\end{equation}

Now let $f_1=e_1$, and $f_2=\tau(f_1)=\tau(e_1)$. Since $f_2\in\spann\{e_2,e_3\}$ and
$\<f_2,e_3\>>0$, we have
$f_2=\cos\theta\, e_2+\sin\theta\, e_3$ for $\theta=\arccos\<f_2,e_2\>\!\in\,]0,\pi[$.
Set $f_3=-\sin\theta\, e_2+\cos\theta\,e_3$ and
$f_4=\tau(f_3)$.
As $f_4\in\{f_1,f_2,f_3\}^\perp=\{e_1,e_2,e_3\}^\perp$, it follows that $f_4=r e_4$, where
$r=\in\{-1,1\}$. Denoting $\underline{f}=(f_1,f_2,f_3,f_4)$, we get
$[\tau]_{\underline{f}}=R_{\frac{\pi}{2}}\!\oplus\!R_{\frac{\pi}{2}}$.
An easy calculation shows that
\begin{align}
  [\tau]_{\underline{e}}&=
  \begin{pmatrix}
    0 & -\cos\theta & -\sin\theta & 0 \\
    \cos\theta & 0 & 0 & r\sin\theta \\
    \sin\theta & 0 & 0 & -r\cos\theta \\
    0 & -r\sin\theta & r\cos\theta & 0
  \end{pmatrix} \label{pretau} \\
  \intertext{and}
  [\s^{-1}\tau]_{\underline{e}}&=
  \begin{pmatrix}
    \cos\theta & 0 & 0 & r\sin\theta \\
    0 & \cos\theta & \sin\theta & 0 \\
    0 & -r\sin\theta & r\cos\theta & 0 \\
    -\sin\theta & 0 & 0 & r\cos\theta
  \end{pmatrix}. \nonumber
\end{align}

If $r=-1$, then $u=\sin\theta\,e_1+(\cos\theta-1)e_2$ is an eigenvector of $\s^{-1}\tau$
with eigenvalue $1$. This means that $\s(u)=\tau(u)$, and thus that $\spann\{u,\s(u)\}$ is
invariant under both $\s$ and $\tau$, which contradicts the irreducibility of the
representation $(\s,\tau)$. Hence $r=1$.

For any unit vector $v=\sum_{i=1}^4v_ie_i\in V$, we have
\begin{multline*}
  \<\s(v),\tau(v)\>=\<v,\s^{-1}\tau(v)\>=\sum_{i,j=1}^4v_iv_j\<e_i,\s^{-1}\tau(e_j)\>=\\
  =\sum_{i=1}^4v_i^2\<e_i,\s^{-1}\tau(e_i)\>=\|v\|^2\cos\theta=\cos\theta=
  \<\s(e_1),\tau(e_1)\>.
\end{multline*}
Hence, the number $\theta=\arccos\<\s(e_1),\tau(e_1)\>$ is an invariant for the
representation $(\s,\tau)$.
On the other hand, it completely determines the representation.
The identity (\ref{pretau}), with $r$ now specified to $1$, can be rewritten as 
\begin{equation} \label{tau}
  [\tau]_{\underline{e}}= T_\theta(R_{\frac{\pi}{2}}\oplus R_{\frac{\pi}{2}})T_{-\theta}
\end{equation}
where $T_\theta$ is defined by (\ref{ttheta}). This, together with (\ref{sigma}), implies
that 
$\{(R_{\frac{\pi}{2}}\oplus R_{\frac{\pi}{2}},T_\theta(R_{\frac{\pi}{2}}\oplus
R_{\frac{\pi}{2}})T_{-\theta})\}_{\theta\in]0,\pi[}$ classifies
$(\mathfrak{IR}_o)_4$. Lemma~\ref{rho} now gives the result for $(\mathfrak{IR}_p)_4$,
completing the proof of Proposition~\ref{class}.

\bibliographystyle{plain}
\bibliography{../litt.bib}

\end{document}